\documentclass{amsart}

\usepackage{amsmath, amssymb,graphicx,mathrsfs,enumerate}
\usepackage[all]{xy}

\usepackage{latexsym}

\newtheorem{proposition}{Proposition}[section]
\newtheorem{lemma}[proposition]{Lemma}
\newtheorem{corollary}[proposition]{Corollary}
\newtheorem{theorem}[proposition]{Theorem}

\newcommand{\cp}{\overline{c}_p}

\begin{document}

\title{On the orders of composition factors in completely reducible groups}

\author{Attila Mar\'oti}
\address{Alfr\'ed R\'enyi Institute of Mathematics, Re\'altanoda utca 13-15, H-1053, Budapest, Hungary}
\email{maroti.attila@renyi.hu}

\author{Saveliy V. Skresanov}
\address{Alfr\'ed R\'enyi Institute of Mathematics, Re\'altanoda utca 13-15, H-1053, Budapest, Hungary}
\email{skresanov.savelii@renyi.hu}

\thanks{The project leading to this application has received funding from the European Research Council (ERC) under the European Union's Horizon 2020
research and innovation programme (grant agreement No 741420). The first author was also supported by the National
Research, Development and Innovation Office (NKFIH) Grant No.~K138596, No.~K132951 and Grant No.~K138828.}

\begin{abstract}
	We obtain an asymptotic upper bound for the product of the $p$-parts of the
	orders of certain composition factors of a finite group acting
	completely reducibly and faithfully on a finite vector space of order
	divisible by a prime $p$. An application is given for the diameter of a
	nondiagonal orbital graph of an affine primitive permutation group.    	
\end{abstract}
\maketitle

\section{Introduction}

The composition length $s(G)$ of a finite group $G$ is the length of any composition series of $G$.
Obtaining bounds for this invariant has been an important area of study in finite group theory.
For instance, Glasby, Praeger, Rosa, Verret~\cite{GPRV} proved that if $G$ is a permutation group of degree $d$ with $r$ orbits, then $s(G) \leq \frac{4}{3} (d-r)$.
In the special case when $G$ is primitive, they~\cite{GPRV} gave a logarithmic bound in $d$ for $s(G)$, namely $s(G) \leq \frac{8}{3} \log_{2}d - \frac{4}{3}$. 

A finite primitive permutation group \( G \) is affine if it has an abelian minimal normal subgroup~\( V \).
The group \( G \) decomposes into a semidirect product \( HV \), where \( H \) is a point stabilizer in \( G \),
moreover, \( H \cap V = 1 \) and the vector space \( V \) may be viewed as an irreducible \( H \)-module. More generally, let $V$ have dimension $n$ over the finite field $\mathbb{F}_{q}$ of size $q$ and let $V$ be a completely reducible, faithful
$\mathbb{F}_{q}H$-module for a finite group $H$. Glasby, Praeger, Rosa, Verret~\cite{GPRV} and Holt, Tracey \cite{HT} gave sharp upper bounds for $s(H)$ of the form $C n \log q$ for explicit constants $C$.
 
The bound can be made more precise if one focuses on cyclic composition factors only.
We continue to assume that $V$ is a completely reducible, faithful $\mathbb{F}_{q}H$-module for a
finite group $H$ with $|V| = q^{n}$ for $q = p^{f}$ with $p$ a prime and $n$, $f$ are integers. Let
$r$ be the number of irreducible summands of $V$. Giudici, Glasby, Li, Verret~\cite{GGLV} proved
that the number of composition factors of $H$ of order $p$ is at most $\frac{\epsilon_{q} n - r}{p-1}$
where $\epsilon_{q}$ is $\frac{4}{3}$ if $p = 2$ and $f$ is even, is $\frac{p}{p-1}$ if $p$
is a Fermat prime (a prime of the form $2^{2^{k}} + 1$ for some integer $k \geq 0$), and is $1$ otherwise.

Some results bound products of orders of special kinds of composition factors in a composition series of a finite group.
For example, Guralnick, the first author, and Pyber~\cite{GMP} investigate products of the orders of abelian or nonabelian
composition factors of a finite group and use these results to classify
primitive permutation groups $A$ and $G$ of degree $d$ with $G$ normal in $A$ and $|A/G| \geq d$.   

In this paper we will also present a bound on the product of the orders of certain composition factors.
For a prime \( p \) and an integer \( N \) let \( v_p(N) \) denote the largest \( k \) such that \( p^k \) divides~\( N \).
Given a finite group \( G \) with composition series \( 1 = G_0 < G_1 < \dots < G_m = G \),
let \( c_p(G) \) be the sum of \( v_p(|G_i/G_{i-1}|) \) for such \( i \in \{ 1, \dots, m \} \) that
\( G_i/G_{i-1} \) is not isomorphic to a finite simple group of Lie type in characteristic~\( p \).
By the Jordan--H\"older theorem \( c_p(G) \) does not depend on the choice of the composition series, so it is an invariant of~\( G \).
Notice that if the group does not contain composition factors isomorphic to finite simple groups of Lie type in characteristic~\( p \)
(if, for instance, the group is \( p \)-solvable), then \( c_p(G) \) is equal to \( v_p(|G|) \).

The following may be viewed as an asymptotic extension of the main theorem of Giudici, Glasby, Li, Verret~\cite{GGLV}.

\begin{theorem}\label{main}
	There exists a universal constant \( C \) such that the following holds.
	Let $q$ be a power of a prime $p$ and let $V$ be a finite vector space of
	dimension $n$ over the field of size $q$. If $H$ is a subgroup of $\mathrm{GL}(V)$
	acting completely reducibly with \( r \) irreducible summands, then
	$$ c_{p}(H) \leq C \cdot \frac{n - r}{p-1}. $$
\end{theorem}	

Note that one cannot hope to obtain a similar bound for \( c_p(H) \) which is linear in~\( n \), for \( p \) fixed, unless one excludes
composition factors isomorphic to finite simple groups of Lie type in characteristic~\( p \) from the definition of \( c_p(H) \).
For instance, \( v_p(|\mathrm{GL}(V)|) = n(n-1)/2 \) if \( V \) has dimension \( n \) over a field of order~\( p \), and this is quadratic in \( n \).

If the linear group $H$ in Theorem~\ref{main} is $p$-solvable, then a good and explicit bound is known for $c_{p}(H) = v_{p}(|H|)$, namely,
Schmid~\cite[p. 211]{S} showed that $c_{p}(H) \leq n p/{(p-1)}^{2}$.
This is related to Brauer's $k(B)$ problem, to the $k(GV)$ theorem, and to the noncoprime $k(GV)$ problem.
For example, Kov\'acs and Robinson~\cite{KR} proved that there exists a universal constant $c$ such that whenever $V$ is a
finite, completely reducible, and faithful $\mathbb{F}_{p}H$-module of dimension $n$ for a finite $p$-solvable group $H$ with a prime $p$,
then the number $k(HV)$ of conjugacy classes of the semidirect product $HV$ is at most $c^{n}|V|$.
It turned out after the proof of the $k(GV)$ theorem that $n\log_{p}c$ can be taken to be $v_{p}(|H|) = c_{p}(H) \leq n p/{(p-1)}^{2}$.

Another motivation to establish Theorem~\ref{main} was a recent work of Robinson~\cite{Robinson} in
which, answering a question of Etingof, he proved similar upper bounds for the index of an abelian
normal subgroup of a $p'$-group contained in $\mathrm{GL}(n,\mathbb{C})$ for any given prime $p$.

For our final motivation, let $G$ be a permutation group acting on a finite set $X$. An orbital graph of $G$ is a graph with
vertex set $X$ whose arc set is an orbit of $G$ on $X \times X$.  An orbital graph whose arcs are a
subset of the diagonal $\{ (x, x) \mid x \in X \}$ is called a diagonal orbital graph. A criterion
of Higman~\cite{Higman} states that a transitive permutation group is primitive if and only if all
its nondiagonal orbital graphs are connected. Liebeck, Macpherson and Tent~\cite{LMT} described
finite primitive permutation groups whose nondiagonal orbital graphs have bounded diameter (we note
that in~\cite{LMT} orbital graphs are considered to be undirected). See also the papers of
Sheikh~\cite{Sh} and Rekv\'enyi~\cite{Re}. 

The paper~\cite{MS} contains upper bounds for the diameters of nondiagonal orbital graphs of affine
primitive permutation groups. Improving on a bound in~\cite{MS}, the second author~\cite{Saveliy}
proved that there exists a universal constant $C$ such that the diameter of a
nondiagonal orbital graph for an affine primitive permutation group $G$ of degree $p^{n}$, for a
prime $p$ and an integer $n$, is at most $C n^{3}$, provided that a point-stabilizer of $G$ has order divisible by $p$. 

As an application of Theorem~\ref{main}, we obtain a strong upper bound for the orbital diameter of
an affine primitive permutation group $G$ with point-stabilizer $H$, under the condition that
$c_{p}(H) \geq 1$, where $p$ is the prime dividing the degree of $G$.      

\begin{corollary}\label{cor}
	There exists a universal constant $C$ such that whenever $G$ is an affine primitive permutation group of degree $p^{n}$,
	where $p$ is a prime and $n$ is an integer, with a point-stabilizer $H$ satisfying $c_{p}(H) \geq 1$,
	then the diameter of any nondiagonal orbital graph of $G$ is less than $C n^{2}/c_{p}(H)$.  
\end{corollary}

Note that if the composition factors of $H$ belong to a list of known finite simple groups,
then Corollary~\ref{cor} is independent from the classification of finite simple groups. 

\section{Bounds on prime divisors of the orders of finite simple groups}

The purpose of this section is to establish Theorem~\ref{main} in the special case when $H$ is a
quasisimple group acting irreducibly on $V$. The main result of the section is Proposition~\ref{simpbound}.

The proof relies on bounds for prime divisors of the orders of finite simple groups of Lie type.
Similar results have been obtained in~\cite{Artin, Buen, MSim}, but we will require finer bounds in terms of the dimensions
of irreducible projective modules of groups of Lie type.

We need the following corollary of a result of Artin~\cite{Artin}.

\begin{lemma}\label{artin}
	Let \( r \) be a nonnegative integer, and let \( p \) be a prime. If \( a = \pm r \) or \( a = r^2 \) then
	\[ v_p\left( \prod_{i = 1}^m (a^i - 1) \right) \leq 2\frac{\log{(r+1)^m}}{\log p}. \]
\end{lemma}
\begin{proof}
	In~\cite[p.~463]{Artin}, cf.~\cite[Lemma~4.2]{Buen}, it was shown that
	\[
		p^{v_p\left( \prod_{i = 1}^m (a^i - 1) \right)} \leq
		\begin{cases}
			3^{m/2} (r + 1)^m, & \text{ if } r \text{ is even, } a = \pm r \text{ or } a = r^2,\\
			2^{m} (r + 1)^m, & \text{ if } r \text{ is odd, } a = \pm r,\\
			4^{m} (r + 1)^m, & \text{ if } r \text{ is odd, } a = r^2.
		\end{cases}
	\]
	The right-hand side can be bounded above by \( (r+1)^{2m} \). The claim follows by taking base~\( p \) logarithms.
\end{proof}

Our notation for finite simple groups of Lie type follows~\cite{KL}.

\begin{lemma}\label{classical}
	Let \( G \) be \( \mathrm{L}_m(r) \), \( \mathrm{U}_m(r) \), \( \mathrm{PSp}_{2m}(r) \), \( \Omega_{2m+1}(r) \), or \( \mathrm{P\Omega}_{2m}^\pm(r) \).
	If \( p \) is a prime not dividing \( r \), then
	\[ v_p(|G|) \leq 3 \frac{\log (r+1)^m}{\log p}. \]
\end{lemma}
\begin{proof}
	We use Lemma~\ref{artin} with \( a = r \) for linear groups, \( a = -r \) for unitary groups, and \( a = r^2 \) for the spinor and orthogonal groups;
	see~\cite[Table~5.1.A]{KL} for the order formulae for these groups.
	For all cases except of the orthogonal groups in even dimension that gives us the bound
	\[ v_p(|G|) \leq 2\frac{\log (r+1)^m}{\log p}. \]
	In case of \( \mathrm{P\Omega}_{2m}^\pm(r) \), the prime \( p \) may divide \( r^m \pm 1 \) and \( \prod^{m-1}_{i = 1} (r^{2i} - 1) \).
	Since \( v_p(r^m \pm 1) \leq \log{(r+1)^m}/\log p \) we get the final bound.
\end{proof}

\begin{lemma}\label{exceptional}
	Let \( G \) be an exceptional finite simple group defined over the field of order \( r \).
	If \( p \) is a prime not dividing \( r \), then
	\[ v_p(|G|) \leq 30 \frac{\log (r+1)}{\log p}. \]
\end{lemma}
\begin{proof}
	We use the order formulae for the exceptional groups, see~\cite[Table~5.1.B]{KL}.
	For \( {}^2B_2(r) \), \( {}^2G_2(r) \), \( {}^2F_4(r) \), and \( {}^3D_4(r) \) we estimate the \( p \)-part
	of the order from above by \( (r+1)^{16} \), so \( v_p(|G|) \leq 16 \log{(r+1)} / \log p \) in this case.

	For the other groups we use Lemma~\ref{artin} with the following parameters:
	\begin{align*}
		G_2(r),\, a = r^2,\, m = 3,\\
		F_4(r),\, a = r^2,\, m = 6,\\
		E_6(r),\, a = r,\, m = 12,\\
		E_7(r),\, a = r^2,\, m = 9,\\
		E_8(r),\, a = r^2,\, m = 15,\\
		{}^2E_6(r),\, a = -r, m = 12.
	\end{align*}
	Clearly the \( E_8(r) \) case dominates the rest, which gives us the claimed bound.
\end{proof}

The next lemma shows that the dimensions of cross-characteristic modules for a group of Lie type are large in comparison to the prime divisors of the order of the group.

\begin{lemma}\label{pbound}
	There exists a universal constant \( C \) such that the following is true.
	Let \( G \) be a nonabelian finite simple group of Lie type defined over a field of order~\( r \)
	having an irreducible projective representation of dimension \( n \) over a field of characteristic~\( p \).
	If \( p \) divides \( |G| \) and does not divide \( r \), then \( p \leq C\cdot n \).
	Moreover, the following are true:
	\begin{enumerate}[\normalfont(1)]
		\item If \( G \) is \( \mathrm{L}_m(r) \), \( \mathrm{PSp}_{2m}(r) \), \( \mathrm{U}_m(r) \), \( \mathrm{P \Omega}^{\pm}_{2m}(r) \), or \( \Omega_{2m+1}(r) \),
			then \( r^{m-1} \leq C \cdot n \).
		\item If \( G \) is an exceptional group, then \( r \leq C \cdot n \).
	\end{enumerate}
\end{lemma}
\begin{proof}
	Assume first that \( G \) is \( \mathrm{L}_m(r) \), \( \mathrm{PSp}_{2m}(r) \), \( \mathrm{U}_m(r) \), \( \mathrm{P \Omega}^{\pm}_{2m}(r) \), or \( \Omega_{2m+1}(r) \).
	We claim that for every type of the group (linear, symplectic, unitary or orthogonal) the dimension \( n \) is bounded from
	below by \( C_1 \cdot r^{\alpha m + \beta} \) where \( C_1 \) is some universal constant and \( \alpha, \beta \) depend only on the type of the group.
	For example, if \( G \simeq \mathrm{U}_{m}(r) \) and \( m \) is even, then by~\cite[Table~5.3.A]{KL}, we have \( n \geq (r^{m-1} - 1)/(r+1) \).
	Therefore \( n \geq \frac{1}{2} r^{m-1} \), so \( \alpha m + \beta \) is \( m - 1 \) in this case.

	The lower bounds on \( n \) extracted from~\cite[Table~5.3.A]{KL} are collected in the third column of Table~1.
	In the table below we list the expressions \( \alpha m + \beta \) such that \( n \geq C_1 \cdot r^{\alpha m + \beta} \) for classical groups:
	\begin{center}
	\begin{tabular}{c | c c c c c}
		Group & \( L_m(r) \) & \( \mathrm{PSp}_{2m}(r) \) & \( U_m(r) \) & \( \mathrm{P \Omega}^{\pm}_{2m}(r) \) & \( \Omega_{2m+1}(r) \)\\\hline
		Bound & \( m-1 \) & \( m \) & \( m-1 \) & \( 2m - 3 \) & \( 2m-2 \)
	\end{tabular}
	\end{center}
	Clearly, for some constant \( C \), we have \( r^{m-1} \leq C \cdot n \), proving~(1).

	Since \( p \) divides \( |G| \), it divides at least one of the factors from the order formula for \( |G| \), see~\cite[Table~5.1.A]{KL}.
	In the second column of Table~\ref{tabSimp} we list the largest such factors, that is, only those which are not dominated
	by the lower bound on the dimension~\( n \). For instance, if \( G \simeq \Omega_{2m+1}(r) \), then \( p \)
	divides one of \( r^{2i} - 1 \), \( i = 1, \dots, m \). We know that \( n \geq C_1 \cdot r^{2m-2} \) from the table above,
	so \( r^{2i} - 1 \leq C_1' \cdot n \) for \( i = 1, \dots, m-1 \) and some universal constant \( C_1' \).
	Hence we put the factor \( r^{2m} - 1 \) in Table~\ref{tabSimp}.

	Note that \( r^{2m} - 1 \) factorizes as \( (r^m - 1)(r^m + 1) \), so \( p \) divides one of the factors, and therefore
	one has \( p \leq C_1' \cdot n \). Similar factorizations can be used for other classical groups, so
	we derive that \( p \leq C \cdot n \) for some universal constant \( C \). 

	Assume now that \( G \) is an exceptional group of Lie type. The dimension \( n \) can be bounded from below by \( C_2 \cdot r^\alpha \)
	for some universal constants \( C_2 \) and \( \alpha \) depending only on the type of the group by~\cite[Table~5.3.A]{KL}.
	We list the corresponding \( \alpha \) for the exceptional groups in the following table:
	\begin{center}
	\begin{tabular}{c | c c c c c c c c c c}
		Group & \( E_6 \) & \( E_7 \) & \( E_8 \) & \( F_4 \) & \( ^2E_6 \) & \( G_2 \) & \( ^3 D_4 \) & \( ^2 F_4 \) & \( \mathrm{Sz} \) & \( ^2 G_2 \)\\ \hline
		Bound & \( 11 \) & \( 17 \) & \( 29 \) & \( 8 \) & \( 11 \) & \( 3 \) & \( 5 \) & \( 5 \) & \( 1 \) & \( 2 \)
	\end{tabular}
	\end{center}
	It immediately follows that \( r \leq C \cdot n \) for some constant \( C \),  proving~(2).

	The prime \( p \) divides the order of the group and, hence, divides some factor in its order formula, see~\cite[Table~5.1.B]{KL}.
	As in the previous case, in the second column of Table~\ref{tabSimp} we list the largest such factor. Note that for the group \( ^3 D_4(r) \)
	there are two factors not dominated by the lower bound for \( n \).

	We factorize the polynomials from the order formulae in order to obtain a bound of the form \( p \leq C \cdot n \) for some universal constant \( C \).
	For example, if \( G \simeq E_6(r) \) and \( p \) divides \( r^{12} - 1 \), we derive that \( p \) divides one of \( r^6 - 1 \) or \( r^6 + 1 \)
	which is smaller than \( r^{11} \). The only nontrivial cases arise when \( G \) is \( ^3 D_4(r) \),
	\( ^2 F_4(r) \) or \( \mathrm{Sz}(r) \). If \( G \simeq {^3}D_4(r) \) and \( p \) divides \( r^8 + r^4 + 1 \), we use the factorization
	\[ r^8 + r^4 + 1 = (r^4 + r^2 + 1)(r^4 - r^2 + 1), \]
	hence \( p \leq 3 \cdot r^5 \). If \( G \simeq {^2}F_4(r) \) and \( p \) divides \( r^6 + 1 \),
	then we use \( r^6 + 1 = (r^2 + 1)(r^4 - r^2 + 1) \), so \( p \leq 3 \cdot r^5 \). Finally, if \( G \simeq \mathrm{Sz}(r) \)
	and \( p \) divides \( r^2 + 1 \), then recall that \( r = 2^{2e+1} \) for some integer~\( e \) and we have
	\[ r^2 + 1 = (r+1 - \sqrt{2r})(r+1 + \sqrt{2r}). \]
	Therefore \( p \leq 3 \cdot r \) in this case, finishing the proof of the lemma.
\end{proof}

Notice that in the setting of the lemma we also have bounds of the form \( p-1 \leq C'(n-1) \),
\( r^{m-1} - 1 \leq C'(m-1) \) in case~(1), and \( r-1 \leq C'(n-1) \) in case~(2) for some universal
constant \( C' \).

\begin{table}[h]
\begin{tabular}{l | l | p{5cm}}
	Group & Largest factors & Lower bounds\\ \hline
	\( L_2(r) \) & \( r^2 - 1 \) & \( (r-1)/\gcd(2, r-1) \)\\
	\( L_m(r), \, m \geq 3 \) & \( r^m - 1 \) & \( r^{m-1} - 1 \)\\
	\( \mathrm{PSp}_{2m}(r), \, m \geq 2 \) & \( r^{2i} - 1 \), \( m < 2i \leq 2m \) &
	\( (r^m - 1)/2,\, r \text{ odd} \) \newline \( r^{m-1}(r^{m-1} - 1)(r-1)/2,\, r \text{ even} \)\\
	\( U_m(r), \, m \geq 3 \) & \( r^m - (-1)^m \) &
	\( r(r^{m-1} - 1)/(r+1),\, m \text{ odd} \) \newline  \( (r^m - 1)/(r+1),\, m \text{ even} \)\\
	\( \mathrm{P\Omega}^+_{2m}(r), \, m \geq 4 \) & \( r^{2m-2} - 1 \) &
	\( (r^{m-1} - 1)(r^{m-2} + 1),\, r \neq 2, 3, 5 \) \newline \( r^{m-2}(r^{m-1}-1),\, r = 2, 3, 5 \)\\
	\( \mathrm{P\Omega}^-_{2m}(r), \, m \geq 4 \) & \( r^{2m-2} - 1 \) & \( (r^{m-1} + 1)(r^{m-2} - 1) \)\\
	\( \mathrm{\Omega}_{2m+1}(r), \, m \geq 3 \), \( r \) odd & \( r^{2m} - 1 \) &
	\( r^{2m-2} - 1,\, r > 5 \) \newline \( r^{m-1}(r^{m-1}-1),\, r = 3, 5 \)\\ \hline
	\( E_6(r) \) & \( r^{12} - 1 \) & \( r^9(r^2 - 1) \)\\
	\( E_7(r) \) & \( r^{18} - 1 \) & \( r^{15}(r^2 - 1) \)\\
	\( E_8(r) \) & \( r^{30} - 1 \) & \( r^{27}(r^2 - 1) \)\\
	\( F_4(r) \) & \( r^{12} - 1 \) & \( r^6(r^2 - 1) \), \( r \) odd \newline \( r^7(r^3 - 1)(r-1)/2 \), \( r \) even\\
	\( {}^2E_6(r) \) & \( r^{12} - 1 \) & \( r^9(r^2 - 1) \)\\
	\( G_2(r) \) & \( r^6 - 1 \) & \( r(r^2 - 1) \)\\
	\( {}^3D_4(r) \) & \( r^8 + r^4 + 1 \), \( r^6 - 1 \) & \( r^3(r^2 - 1) \)\\
	\( {}^2F_4(r) \) & \( r^6 + 1 \) & \( r^4\sqrt{r/2}(r - 1) \)\\
	\( \mathrm{Sz}(r) \) & \( r^2 + 1 \) & \( \sqrt{r/2}(r - 1) \)\\
	\( {}^2G_2(r) \) & \( r^3 + 1 \) & \( r(r - 1) \)\\
\end{tabular}
\medskip
\caption{Largest factors in order formulae and lower bounds of dimensions of representations for groups of Lie type}\label{tabSimp}
\end{table}

The following result will be used in the main proof. Recall that \( G \) is quasisimple, if it is perfect and \( G/Z(G) \) is nonabelian simple.

\begin{proposition}\label{simpbound}
	There exists a universal constant \( C  \) such that the following is true.
	Let \( G \) be a quasisimple group such that \( G/Z(G) \) is not isomorphic to a group of Lie type in characteristic \( p \).
	If \( G \) has an irreducible projective representation of dimension \( n \) over a field of characteristic~\( p \), then
	\[ v_p(|G|) \leq C \cdot \frac{n-1}{p-1}. \]
\end{proposition}
\begin{proof}
	By~\cite[Corollary~5.3.3]{KL}, the degree of a minimal projective \( p \)-modular representation of \( G \)
	is bounded below by the corresponding number for \( G/Z(G) \). We may thus replace \( G \) by \( G/Z(G) \) and assume that \( G \) is simple.

	Let \( C \) be a large fixed constant (how to specify \( C \) will be clear from the proof).
	Notice that by choosing \( C \) large enough we may assume that \( G \) is not a sporadic group.

	If \( G \) is isomorphic to \( \mathrm{Alt}(m) \), \( m \geq 5 \), then by~\cite[Proposition~5.3.7]{KL} one has \( n \geq m-4 \). Thus by Legendre's formula
	\[ v_p(|\mathrm{Alt}(m)|) \leq \frac{m-1}{p-1} \leq 5 \frac{n-1}{p-1}, \]
	where the last inequality uses the fact that \( n+3 \leq 5(n-1) \) for \( n \geq 2 \). Now the claimed inequality follows for \( C \geq 5 \).

	Now we assume that \( G \) is a group of Lie type not in characteristic \( p \).
	We first consider classical groups. Fix \( r \) and \( m \) as in Lemma~\ref{classical}, and notice
	that for \( r \geq 2 \) and \( m \geq 2 \) we have \( (r+1)^m \leq 9(r^{m-1}-1)^2 \).
	Lemma~\ref{classical} implies
	\[ v_p(|G|) \leq 3 \frac{\log (9(r^{m-1}-1)^2)}{\log p}. \]
	If \( p \geq \sqrt{3(r^{m-1}-1)} \), then
	\[ v_p(|G|) \leq 3 \frac{\log(9(r^{m-1}-1)^2)}{\log{\sqrt{3(r^{m-1}-1)}}} = 12 \leq C \cdot \frac{n-1}{p-1}, \] 
	where the last inequality holds for \( C \) large enough by Lemma~\ref{pbound}. If \( p < \sqrt{3(r^{m-1}-1)} \), then
	\[ v_p(|G|) \leq 3 \frac{\log(9(r^{m-1}-1)^2)}{\log 2} < C_1 \sqrt{r^{m-1}-1} < C_2 \cdot \frac{r^{m-1}-1}{p-1}, \]
	for some constants \( C_1, C_2 \). By Lemma~\ref{pbound}~(1), we have \( r^{m-1} - 1 \leq C_3 \cdot (n-1) \) for some \( C_3 \).
	Therefore
	\[ v_p(|G|) \leq C_2 \cdot C_3 \cdot \frac{n-1}{p-1} \leq C \cdot \frac{n-1}{p-1}, \]
	whenever \( C \geq C_2 \cdot C_3 \).

	We turn to the exceptional groups. If \( r \) is the order of the defining field, then \( r+1 \leq 3(r-1) \) and Lemma~\ref{exceptional} imply
	\[ v_p(|G|) \leq 30 \frac{\log (r+1)}{\log p} \leq 30 \frac{\log 3(r-1)}{\log p}. \]
	If \( p \geq \sqrt{3(r-1)} \), then
	\[ v_p(|G|) \leq 30 \frac{\log 3(r-1)}{\log \sqrt{3(r-1)}} = 60 \leq C \cdot \frac{n-1}{p-1}, \]
	where the last inequality uses Lemma~\ref{pbound}. If \( p < \sqrt{3(r-1)} \), then
	\[ v_p(|G|) \leq 30 \frac{\log 3(r-1)}{\log 2} < C_1' \sqrt{r-1} < C_2' \cdot \frac{r-1}{p-1}, \]
	for some constants \( C_1', C_2' \).
	By Lemma~\ref{pbound}~(2), we have \( r-1 \leq C_3' \cdot (n-1) \), hence
	\[ v_p(|G|) \leq C_2' \cdot C_3' \cdot \frac{n-1}{p-1} \leq C \cdot \frac{n-1}{p-1}, \]
	for \( C \geq C_2' \cdot C_3' \).
\end{proof}

\section{Nonabelian composition factors}

For a finite group \( G \) with composition series \( 1 = G_0 < \dots < G_m = G \) let \( \cp(G) \) be the
sum of \( v_p(|G_i/G_{i-1}|) \) over such \( i \in \{ 1, \dots, m \} \) that \( G_i/G_{i-1} \)
is nonabelian and not isomorphic to a finite simple group of Lie type in characteristic~\( p \).
The main result of~\cite{GGLV} bounds the number of composition factors isomorphic to cyclic groups of order~\( p \),
so in order to bound \( c_p(G) \) we may focus on bounding \( \cp(G) \) first.

\begin{proposition}\label{nonabelian}
	There exists a universal constant \( C \) such that the following holds.
	Let $q$ be a power of a prime $p$ and let $V$ be a finite vector space of
	dimension $n$ over the field of size $q$. If $H$ is a subgroup of $\mathrm{GL}(V)$
	acting completely reducibly with \( r \) irreducible summands, then
	$$ \cp(H) \leq C \cdot \frac{n - r}{p-1}. $$
\end{proposition}	
\begin{proof}
	Let $H \leq \mathrm{GL}(V)$ be a counterexample to the statement of the theorem with $n \geq 2$ minimal.
	Under this condition, assume that $|H|$ is as small as possible.
	The proof proceeds in several steps; we choose the constant \( C = \max \{ 20/3,\, C_1 \} \), where \( C_1 \) is the constant~\( C \) from Proposition~\ref{simpbound}.

	\noindent\textbf{Step 1: \( H \) acts irreducibly on \( V \).}
	Assume that $W$ is a nonzero proper irreducible submodule of $V$. Let $K$ be the centralizer of $W$ in $H$. The factor
	group $H/K$ acts irreducibly and faithfully on $W$. Thus $ \cp(H/K) \leq C \cdot \frac{m - 1}{p-1} $
	where $m$ is the dimension of $W$ over the field of size $q$.
	Since $H$ acts completely reducibly on $V$, there exists a submodule $U$ of
	$V$ such that $V = W \oplus U$. The group $K$ acts faithfully on $U$. Since $K$ is
	normal in $H$, it acts completely reducibly on $U$ by Clifford's theorem. By the
	minimality of $n$ again, we have $ \cp(K) \leq C \cdot \frac{(n-m) - (r-1)}{p-1} $.
	These give
	$$ \cp(H) = \cp(H/K) + \cp(K) \leq C \cdot \frac{m - 1}{p-1} + C \cdot \frac{(n-m) - (r-1)}{p-1} = C \cdot \frac{n - r}{p-1}, $$
	a contradiction to the minimality of~\( H \).

	\noindent\textbf{Step 2: \( H \) is perfect.}
	Since \( H \) acts irreducibly on $V$, its derived subgroup \( [H, H] \) acts completely reducibly.
	Now, \( \cp(H) = \cp([H, H]) \) and we may assume that \( H = [H, H] \) by the minimality of \( |H| \).

	\noindent\textbf{Step 3: \( H \) acts primitively on \( V \).}
	Assume that $H$ acts imprimitively on $V$, that is, $H$ preserves a decomposition $V = V_{1} + \ldots + V_t$ of the vector space
	$V$ to (proper) subspaces $V_{i}$ of the same size where $1 \leq i \leq t$ for some integer $t > 1$.
	Let the kernel of the action of $H$ on $\{ V_{1}, \ldots , V_{t} \}$ be $B$.
	We have $ \cp(H/B) \leq (t-1)/(p-1)$ by considering the $p$-part of $t!$. Since $B$ is a proper normal subgroup of $H$, we have $ \cp(B) \leq C \cdot \frac{n - t}{p-1} $. These give
	$$ \cp(H) = \cp(H/B) + \cp(B) \leq \frac{t-1}{p-1} + C \cdot \frac{n - t}{p-1} \leq C \cdot \frac{n - 1}{p-1},$$
	where $C \geq 1$ is used. 

	\noindent\textbf{Step 4: \( H \) acts absolutely irreducibly on \( V \).}
	Let $E = \mathrm{End}_{H}(V)$. This is a field extension of the field of order $q$. Let the order of $E$ be $q^{e}$.
	The group $H$ may be viewed as a subgroup of $\mathrm{GL}(V)$ where $V$ is the
	vector space of dimension $n/e$ over the field $E$. The $EH$-module $V$ remains
	irreducible. Let $e > 1$. The minimality of $n$ gives
	$$ \cp(H) \leq C \cdot \frac{n/e - 1}{p-1} < C \cdot \frac{n - 1}{p-1}. $$
	A contradiction. 

	\noindent\textbf{Step 5: \( H \) does not preserve any proper field extension.}
	Assume that $H$ preserves a field extension structure on $V$ over the field of order $q^e$ for some $e > 1$. The group $H$ may
	be embedded in $\mathrm{GL}(n/e,q^{e}).e$ and since \( H \) is perfect, \( H \) lies in \( \mathrm{GL}(n/e, q^e) \).
	By the argument in~\cite[p. 1028]{GMP}, the group $H$ acts irreducibly (and faithfully) on $V$ viewed as a vector space of dimension $n/e$
	over the field with $q^e$ elements. These give
	$$ \cp(H) \leq C \cdot \frac{n/e - 1}{p-1} < C \cdot \frac{n - 1}{p-1}. $$
	A contradiction.   

	\noindent\textbf{Step 6: The group $H$ is quasisimple.}
	By the argument in~\cite[p. 1029]{GMP}, for every normal subgroup $R$ of $H$ every irreducible constituent of the $R$-module $V$ is absolutely irreducible.

	Since $H$ acts primitively on $V$, every normal subgroup of $H$ acts homogeneously on $V$ by Clifford's theorem.
	In particular, every abelian normal subgroup of $H$ is cyclic by Schur's lemma and is central by the previous paragraph.

	Let $R$ be a normal subgroup of $H$ minimal subject to being noncentral. The center $Z(R)$ of $R$
	is contained in $Z(H)$ and $R/Z(R)$ is characteristically simple. As in the proof
	of~\cite[Theorem 4.1]{GMP}, the group $R$ is either a central product of say $t$
	quasisimple groups $Q_i$ (with the $Q_{i}/Z(Q_{i})$ all isomorphic) or $R/Z(R)$ is
	an elementary abelian $r$-group for some prime $r$. In the second case $R$ is an
	$r$-group with $r$ different from $p$ and it may be proved that $R$ is of
	symplectic type with $|R/Z(R)| = r^{2a}$ for some integer $a$. 

	We follow the proof of~\cite[Theorem 4.1]{GMP} and introduce some notation.
	Let $J_{1}, \ldots , J_{k}$ denote the distinct normal subgroups of $H$ that are minimal with respect to being
	noncentral in $H$. Let $J = J_{1} \cdots J_{k}$ be the central product of these
	subgroups. Let $W$ be an irreducible constituent of the $J$-module $V$. Then $W =
	U_{1} \otimes \cdots \otimes U_{k}$ where $U_i$ is an irreducible $J_{i}$-module.
	If $J_i$ is the central product of $t$ copies of a quasisimple group, then
	$\mathrm{dim}(U_{i}) \geq 2^t$ and if $J_i$ is of symplectic type with
	$J_{i}/Z(J_{i})$ of order $r^{2a}$, then $\mathrm{dim}(U_{i}) = r^a$. 

	The group $H/(Z(H)J)$ embeds into the direct product of the outer automorphism groups of
	the $J_{i}$. Let $J_i$ be a central product of say $t$ quasisimple groups $Q$. The outer
	automorphism group $\mathrm{Out}(J_{i})$ in this case may be viewed as a subgroup of
	$\mathrm{Out}(Q/Z(Q)) \wr \mathrm{Sym}(t)$. Since $\mathrm{Out}(Q/Z(Q))$ is solvable by
	Schreier's conjecture,
	$$ v_{p}( | \mathrm{Out}(J_{i}) / \mathrm{Sol}( \mathrm{Out}(J_{i}) ) | ) \leq v_p(|\mathrm{Sym}(t)|) \leq \frac{t-1}{p-1},$$
	where $\mathrm{Sol}(X)$ denotes the solvable radical of a finite group $X$.
	Now let $J_i$ be a group of symplectic type with $|J_{i}/Z(J_{i})| = r^{2a}$ for some prime $r$ and integer $a$.
	In this case $\mathrm{Out}(J_{i})$ may be viewed as a subgroup of $\mathrm{Sp}_{2a}(r)$ and so
	$$ v_{p}(|\mathrm{Out}(J_{i})|) \leq v_p(|\mathrm{Sp}_{2a}(r)|),$$ which is at most $\frac{(4/3)r^{a}-1}{p-1}$ by~\cite[(3)]{GGLV}.

	Since $n = \dim(V) \geq \dim(W) = \prod_{i} \dim(U_{i}) \geq \sum_{i} \dim(U_{i})$, we have
	\begin{equation}
	\label{e1}	
	\cp(H/(Z(H)J)) \leq \frac{(4/3)n-1}{p-1} \leq \frac{5}{3} \cdot \frac{n-1}{p-1}
	\end{equation}
	by the previous paragraph and the fact that \( n \geq 2 \).

	We claim that exactly one of the $J_{i}$ is nonsolvable with a nonabelian composition
	factor of order divisible by $p$ but different from a group of Lie type in characteristic~$p$.
	Suppose otherwise. If there is no such $J_i$, then $\cp(Z(H)J) = 0$ and so
	\begin{equation}
	\label{e2}
	\cp(H) \leq \cp(H/(Z(H)J)) + \cp(Z(H)J) \leq \frac{5}{3} \cdot \frac{n-1}{p-1} < C \cdot \frac{n-1}{p-1}, 
	\end{equation}
	by (\ref{e1}) and the fact that $C \geq 5/3$, a contradiction. Let the number of such $J_i$ be $m > 1$. Without
	loss of generality, let these be $J_{1}, \ldots , J_{m}$. We have $ \cp(Z(H)J) = \sum_{i=1}^{m} \cp(J_{i})$.
	For each $i$ with $1 \leq i \leq k$, let $\mathrm{dim}(U_{i}) = n_i$. By the minimality of $n$, we find that
	$$\sum_{i=1}^{m} \cp(J_{i}) \leq C \cdot \frac{(\sum_{i=1}^{m} n_{i}) - m}{p-1}.$$
	If \( m \geq 3 \) or $m=2$ and $\max \{ n_{1}, n_{2} \} \geq 4$, then \( \sum_{i=1}^{m} n_i \leq \frac{3}{4} \prod_{i=1}^{m} n_i \leq \frac{3}{4}n \), hence
	\[ \cp(H) = \cp(Z(H)J) + \cp(H/(Z(H)J)) \leq \frac{3C}{4} \cdot \frac{n - 1}{p-1} + \frac{5}{3} \cdot \frac{n - 1}{p-1} \leq C \cdot \frac{n - 1}{p-1}, \]
	where the last inequality holds since \( C \geq 20/3 \). A contradiction.

	If \( m = 2 \) and \( \max \{ n_{1}, n_{2}\} \leq 3 \), then \( \cp(\mathrm{Out}(J_i)) = 0 \) for \( i = 1,2 \) and hence
	\[ \cp(H/(Z(H)J)) \leq \frac{(4/3)\sum_{i = 3}^k n_i - 1}{p-1}, \]
	so
	\[ \cp(H) \leq C \cdot \frac{(n_1 + n_2) - 1}{p-1} + \frac{(4/3)\sum_{i = 3}^k n_i - 1}{p-1} \leq C \cdot \frac{n - 1}{p-1}, \] by the minimality of $n$,
	where the last inequality holds since \( C \geq 4/3 \). A contradiction. We thus have \( m = 1 \).

	We claim that $k = 1$. Assume that $k \geq 2$. By the previous paragraph and without loss
	of generality, $\cp(J_{1}) \geq 1$ and $\cp(J_{i}) = 0$ for every $i$ with $2 \leq i \leq k$.
	By the minimality of $n$ and the fact that $k \geq 2$ and \( n_2 \geq 2 \), we have
	$$ \cp(Z(H)J) = \cp(J_{1}) \leq C \cdot \frac{n_1 - 1}{p-1} \leq \frac{C}{2} \cdot \frac{n - 1}{p-1}. $$
	This together with the bound (\ref{e1}) and \( C \geq 10/3 \) give $\cp(H) < C \cdot \frac{n-1}{p-1}$, a contradiction.

	The group $J = J_{1}$ is a central product of say $t$ quasisimple groups $Q_i$ (with the $Q_{i}/Z(Q_{i})$ all isomorphic).
	We claim that $t=1$. Assume for a contradiction that $t \geq 2$. Let $W$ be an irreducible constituent of the $J$-module $V$.
	Then $W = W_{1} \otimes \cdots \otimes W_{t}$ where $W_i$ is an irreducible $Q_{i}$-module for every $i$ with $1 \leq i \leq t$ by~\cite[Lemmas 5.5.5 and 2.10.1]{KL}.
	For each $i$ with $1 \leq i \leq t$, let $m_{i}$ be $\dim(U_{i}) \geq 2$. We have $n \geq \prod_{i=1}^{t} m_{i} \geq \sum_{i=1}^{t} m_{i}$.
	We get $\cp(J) \leq C \cdot \frac{n - t}{p-1}$ by the minimality of $n$ and $\cp( H/(Z(H)J)) \leq \frac{t - 1}{p-1}$ by Schreier's conjecture.
	This is a contradiction since $C \geq 1$. We conclude that $t=1$.    

	Since $H$ is perfect, $H = JZ(H)$ and so $H = J$ is quasisimple.
	Since $H$ acts absolutely irreducibly on \( V \) and is quasisimple, the final contradiction follows from Proposition~\ref{simpbound}.
\end{proof}

\section{Proofs of the main results}

\noindent {\it Proof of Theorem~\ref{main}.}
Let $V$ be a finite vector space of dimension $n$ over the field of size $q$. Let $H$ be a subgroup of $\mathrm{GL}(V)$ acting completely reducibly on $V$. Let $r$ be the number of
irreducible summands of the $H$-module $V$. We claim that $c_{p}(H) \leq C \cdot \frac{n - r}{p-1}$ for some universal constant $C$. 

We prove the bound by induction on $n$. If $n=1$, then the size of $H$ is not divisible by $p$ and so $c_{p}(H) = 0$.
Assume that $n \geq 2$ and that the claim is true for $n-1$. If the $H$-module $V$ contains an irreducible summand $W$ of dimension $1$
and $K$ denotes the centralizer of $W$ in $H$, then $c_{p}(H) = c_{p}(K) \leq C \cdot \frac{(n-1) - (r-1)}{p-1}$ by the induction hypothesis.
We may assume that every submodule of $V$ has dimension at least $2$. In particular, $r \leq n/2$.
The number of composition factors of $H$ isomorphic to the cyclic group of order $p$ is at most $((4/3)n - r)/(p-1)$ by~\cite[Theorem 1]{GGLV}.
This is at most $\frac{8}{3} \frac{n-r}{p-1}$ since $r \leq n/2$. Thus
$$c_{p}(H) \leq \frac{8}{3} \frac{n-r}{p-1} + \cp(H) \leq C \cdot \frac{n-r}{p-1},$$
where $C$ is $8/3$ plus a constant whose existence is assured by Proposition~\ref{nonabelian}. \qed

\noindent {\it Proof of Corollary~\ref{cor}.}
Let $C$ be a constant whose existence is assured by Theorem~\ref{main}. Let $G$ be an affine
primitive permutation group of degree $p^{n}$ where $p$ is a prime and $n$ is an integer.
Let $H$ be a point-stabilizer in $G$ satisfying $c_{p}(H) \geq 1$. The diameter of any nondiagonal orbital
graph of $G$ is at most $(p-1)n$ by~\cite[Proposition 3.2]{MS}. On the other hand, $p-1 \leq C (n-1)/c_{p}(H)$ by Theorem~\ref{main}. \qed

\end{document}